\documentclass[11pt,leqno]{amsart}
\usepackage{amsmath, amssymb, amsfonts, amsthm, graphicx,color}
\usepackage{comment,latexsym,mathrsfs}
\oddsidemargin = - 0.0 cm
\allowdisplaybreaks

\newcommand{\nc}{\newcommand}
\nc{\pdp}{\dot{\phi}_{+}}
\nc{\pdm}{\dot{\phi}_{-}}
\nc{\pd}{\dot{\phi}}
\nc{\ydp}{\dot{y}_{+}}
\nc{\ydm}{\dot{y}_{-}}
\nc{\yodm}{\dot{y_1}_{-}}
\nc{\yodp}{\dot{y_1}_{+}}
\nc{\ytdm}{\dot{y_2}_{-}}
\nc{\ytdp}{\dot{y_2}_{+}}
\nc{\yp}{y_+}
\nc{\ym}{y_-}
\nc{\yd}{\dot{y}}
\nc{\fp}{\phi_{+}}
\nc{\fm}{\phi_{-}}
\nc{\de}{\delta}
\nc{\e}{\epsilon}
\nc{\ds}{\displaystyle}
\nc{\bt}{\beta_2}
\nc{\bo}{\beta_1}
\nc{\mo}{m_1}
\nc{\mt}{m_2}
\nc{\xo}{x_1}
\nc{\xt}{x_2}
\nc{\yo}{y_1}
\nc{\yt}{y_2}
\nc{\yod}{\dot{y}_1}
\nc{\ytd}{\dot{y}_2}
\nc{\dpl}{d_{+}}
\nc{\dmi}{d_{-}}
\nc{\vv}{\mathbf{v}}
\nc{\vn}{\mathbf{n}}
\nc{\Yd}{\dot{Y}}
\nc{\g}{\gamma}
\nc{\ti}{\theta_i}
\nc{\tip}{\theta_i'}

\theoremstyle{plain}
\newtheorem{theorem}{Theorem}[section]
\newtheorem{lemma}[theorem]{Lemma}

\newtheorem{rmk}[theorem]{Remark}

\theoremstyle{definition}

\newtheorem{corollary}[theorem]{Corollary}

\begin{document}

\title{
 Billiard dynamics of  bouncing dumbbell }

\author{Y. Baryshnikov, V. Blumen, K. Kim, V. Zharnitsky }

\address{Department of Mathematics, University of Illinois, Urbana, IL 61801}

\begin{abstract}
A system of two masses connected with a weightless rod (called dumbbell in this paper) interacting with a flat boundary is considered. The sharp bound on the number of collisions with the boundary is found using billiard techniques. In case, the ratio of masses is large and the dumbbell rotates fast, an adiabatic invariant is obtained.
\end{abstract}

\maketitle
\section{Introduction}
 
Coin flipping had been already known to ancient Romans as a way to decide an outcome \cite{telegraph}. More recently,  scientists inspired  by this old question, how unbiased the real (physical) coin is, have been studying coin dynamics, see e.g.  \cite{keller, mahadevan, diaconis}. 

Previous  studies have mainly focused on the dynamics of the flying coin assuming that  it does not bounce and finding the effects of angular momentum 
on the final orientation. Partial analysis in combination  with numerical simulations of the bouncing effects has been done by Vulovic and Prange 
\cite{vulovic}. It appears that this is the only reference that addressed the effect  of bouncing on coin tossing.

On the other hand, there is a well developed theory of mathematical billiards: classical dynamics of a particle moving inside a bounded domain.
The particle moves along straight line until it hits the boundary. Next, the particle reflects from the boundary according to the Fermat's law.

 The billiard  problem originally appeared in the context of Boltzman ergodic hypothesis \cite{sinai} to verify physical assumptions about 
ergodicity  of a gas of elastic spheres. 
However, various techniques in billiard dynamics turned out to be useful beyond the original physical problem. The so-called unfolding technique 
(which is used in this paper) allows one to obtain estimates on the maximal number of bounces of a particle in a wedge. 
One could expect that the bouncing coin dynamics could be interpreted as a billiard ball problem.  

In this paper we consider a simpler system (with fewer degrees of freedom) which we call the  dumbbell. The bouncing coin on a flat surface, restricted to have axis of rotation pointing in the same direction, can be modeled as a system of two masses connected with a weightless 
rod. 

The dumbbell dynamics that is studied in this article is a useful model to initiate investigation of  this potentially useful relation.

Another motivation for the dumbbell dynamics comes from robotics exploratory problems, see {\em e.g.}  \cite{lavalle}. 
Consider an automated system that moves in a bounded domain and interacts with the boundary according to some simple laws. 
In many applications, it is important to cover  the whole region as {\em e.g.} in automated vacuum cleaners such as Roomba. 
Then, a natural question arises:  {\em what simple mechanical system can generate a dense coverage of a certain subset of the given 
configuration space}. 
The dumbbell, compared to a material point, has an extra degree of freedom which can generate more chaotic behavior as {\em e.g.} in Sinai billiards. 
Indeed, a rapidly rotating dumbbell will quickly ``forget'' its initial orientation before the next encounter with the boundary  
raising some hope for stronger ergodicity. 

In this paper, we study the interaction of a dumbbell  with the flat boundary. This is an important first step before understanding 
the full dynamics of the dumbbell in some simple domains.  By appropriately  rescaling the variables, we obtain an associated single particle 
billiard problem with the boundary corresponding to the collision curve (which is piecewise smooth) in the configuration space.  
The number of collisions of the dumbbell with the boundary before scattering out depends on the mass ratio $m_1/m_2$. If this ratio 
is far from 1, then the notion of adiabatic invariance can be introduced as there is sufficient time scales separation. 
We prove an adiabatic invariant type theorem and we describe under what conditions it can be used. 

Finally, we estimate the maximal number of bounces of the dumbbell with the flat boundary. \\

\noindent
{\bf Notation:}
We use some standard notation when dealing with asymptotic expansions in order to avoid cumbersome use of implicit constants. \\
$  f \lesssim g  \Leftrightarrow    f= O(g)  \Leftrightarrow f \leq Cg $ for some $C>0$\\
$f \gtrsim g  \Leftrightarrow g = O(f)$ \\
$f \sim g   \Leftrightarrow    f \lesssim g \,\, {\rm and} \,\, f \gtrsim g    $ \\

\section{Collision Laws}

\subsection{Dumbbell-like System}

Let us consider a dumbbell-like system, which consists of two point masses $\mo$, $\mt$, connected by weightless rigid rod of length $1$ in the two-dimensional space with coordinates $(x,y)$. The coordinates of $\mo$, $\mt$, and the center of mass of the system are given by $(\xo, \yo)$, $(\xt, \yt)$, and $(x, y)$, respectively. Let $\phi$ be the angle measured in the counterclockwise direction from the base line through $\mo$ and to the rod. We also define the mass ratios $\ds \bo = \frac{\mo}{\mo+\mt}$ and $\ds \bt=\frac{\mt}{\mo+\mt}$ which correspond to the distance from the center of mass to $\mt$ and to $\mo$, respectively. 
\begin{figure}[h]
\centering
 \includegraphics[width=3 in]{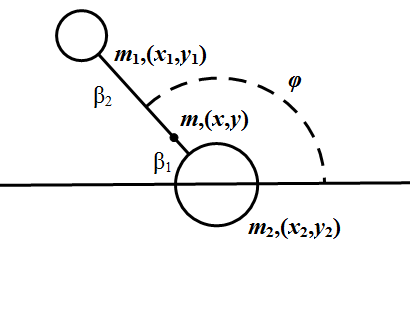}
 \caption{The system of dumbbell.}
\end{figure}

The dumbbell moves freely in the space until it hits the floor. In this system, the velocity of the center of mass in $x$ direction is constant since there is no force acting on the system in $x$ direction. Thus, we may assume without loss of generality that the center of mass does not move in $x$ direction. With this reduction, the dumbbell configuration space is two dimensional with the natural choice of coordinates $(y,\phi)$.

The moment of inertia of the dumbbell is given by 
\[I=\mo \bt^2 + \mt \bo^2 = \bo\bt (\mo + \mt). \]
Introducing the total mass $m = \mo + \mt$, we can write the kinetic energy of the system as 
\begin{align}\label{eq:egy}
K= \frac{1}{2}m\yd^2 + \frac{1}{2}\bo\bt m \pd^2.
\end{align}

Using the  relations, 
\[
\yo = y + \bt \sin \phi
\]
\[
\yt = y - \bo \sin \phi,
\]
we find the velocities of each mass
\[
\yod = \yd + \bt  \pd \cos \phi
\]
\[
\ytd = \yd - \bo \pd \cos \phi.
\]


\subsection{Derivation of Collision Laws}


By rescaling  $y=\sqrt{\frac{I}{m}}Y$, we rewrite the kinetic energy 

\[
K=\frac{m}{2}\yd^2 + \frac{I}{2}\pd^2=\frac{I}{2}\left( \Yd^2 + \pd^2 \right).
\]

By Hamilton's principle of least action, true orbits extremize
\[
\int_{t_0, Y_0, \phi_0}^{t_1, Y_1, \phi_1} K( \Yd, \pd)  \mathrm{d}t.
\]

Since the kinetic energy is equal to the that  of the free particle, 
the trajectories are straight lines between two collisions. When the dumbbell hits the boundary, the collision law is the same as in the classical billiard  since in $(Y, \phi)$ coordinates the action is the same. Using the relations
\begin{equation*}\begin{split}\\
y_1 &=y + \bt\sin \phi \geq 0\\
y_2 &=y - \bo\sin \phi \geq 0,\\
\end{split}
\end{equation*}

we find the boundaries for the dumbbell dynamics in the  $Y$-$\phi$ plane:
\begin{equation}\tag{2a}\label{bd1}
Y= -\sqrt{\frac{m}{I}}\bt\sin\phi =-\sqrt{\frac{\bt}{\bo}}\sin \phi 
\end{equation}

\begin{equation}\tag{2b}\label{bd2}
Y= \;\;\, \sqrt{\frac{m}{I}}\bo\sin\phi=\;\;\,\sqrt{\frac{\bo}{\bt}} \sin \phi. 
\end{equation}

The dumbbell hits the floor if one of the above inequalities becomes an equality. Therefore, we take the maximum of two equations to get the boundaries: 
\begin{align}\tag{2c}\label{bd}
Y={\rm max} \left\{-\sqrt{\bt/\bo}\sin\phi, \sqrt{\bo/\bt}\sin\phi \right\} \text{ for } \phi \in [0, 2\pi]. 
\end{align}


Note that this boundary has non-smooth corner at $\phi = 0, \pi$. This is the case when the dumbbell's two masses hit the floor at the same time. We will not consider this degenerate case in our paper. 

Now we will derive the collision law for the case when only $\mo$ hits the boundary.  We recall that given vector $\vv_-$ and a unit vector $\vn$ the reflection of  $\vv_-$ across  $\vn$ is given by 

\begin{equation}\label{eq:ref}\tag{3}
\vv_{+} = -2 \frac{\vv_-\cdot \vn }{\vn \cdot \vn} \vn + \vv	_-.
\end{equation}

Here and in the remainder of the paper, $x_-, y_-, ...$ are defined as the 
corresponding values right 
before the collision and $x_+, y_+, ...$ are defined as the corresponding values
right before the next collision.

According to the collision law, the angle of reflection is equal to the angle of incidence.  In our case, $\vn$ is the normal vector to the boundary 
\[
Y =\ds -\sqrt{\frac{m}{I}}\bt \sin \phi
\]
 so that 
\begin{equation*}\begin{split}\\
\vn &=\left[1,\sqrt{m/I} \, \bt \cos \phi\right]\\
 \vv_-&= \left[\Yd_-, \pdm\right]=\left[\sqrt{m/I} \, \ydm, \pdm\right].\end{split}
\end{equation*}

Then, using (\ref{eq:ref}), we compute $\ds \vv_{+} = \left[\Yd_+, \pdp\right]$. In this way, we express the translational and the angular velocities after the collision in terms of the velocities before $\mo$ hits the floor. Changing back to the original coordinates, we have

\begin{align}\label{eq:law}\tag{4}
\left( \begin{array}{r}
\ydp \\
\\
\pdp\\
 \end{array} \right)
&=
 \left( \begin{array}{c}
\ds \sqrt{\frac{I}{m}} \Yd_+ \\
\\
\pdp\\
 \end{array} \right) 
=\notag
\left( \begin{array}{l}
\ds \ydm \left ( -1 + \frac{2 \bt\cos^2 \phi}{\bo + \bt \cos^2\phi} \right ) - \pdm \left( \frac{2 \bo\bt\cos \phi}{\bo + \bt \cos^2\phi}  \right)\\
\\
\ds \pdm \left( 1 - \frac{2 \bt\cos^2 \phi}{\bo + \bt \cos^2\phi} \right) - \ydm \left ( \frac{2\cos\phi}{\bo + \bt \cos^2 \phi} \right)
\end{array}\right). 
\end{align}

\vspace{2mm}

\begin{rmk}
The bouncing law for the other case, when $\mt$ hits the boundary can be obtained in a similar manner: we switch $\bo$ and $\bt$, replace $\cos \phi$ and $\sin \phi$ with $-\cos \phi$ and $-\sin \phi$, and replace $\yo$ with $\yt$. 
\end{rmk}

\section{Adiabatic Invariant}

Consider the case when $\mo \ll \mt$ and $\mo$ rotates around $\mt$ with high angular velocity $\pd$ and assume that the center of mass  has slow downward velocity compared to $\pd$.  Since multiplying velocities $(\dot \phi, \dot y)$ by a constant does not change the orbit, we  normalize $\dot \phi$ to be of order 1, then $\dot y$ is small. Consider such dumbbell
slowly approaching the floor, rotating with angular velocity of order 1, {\em i.e.} 
$\dot \phi \sim 1$. 

At some moment the small mass  $\mo$ will hit the floor. If the angle $\phi = \pi/2$ 
(or sufficiently close to it), then the dumbbell will bounce away without experiencing any more collisions. This situation is rather exceptional.

 A simple calculation shows that $|\phi-3\pi/2|$ will be generically  of order $\sqrt{|\dot y|}$ for our  limit $\dot y \rightarrow 0$. In this section we assume this favorable scenario. 
For the corresponding set of initial conditions, we obtain an adiabatic invariant (nearly conserved quantity).  We start by deriving approximate map between two consecutive bounces.

\begin{figure}[h]\label{fig2}
\centering
 \includegraphics[width=4 in]{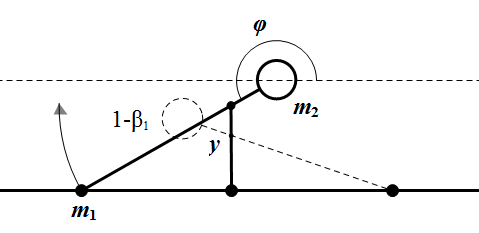}
 \caption{The light mass bounces many times off  the floor while the large 
mass slowly approaches the floor.}
\end{figure}


\begin{lemma}
\label{lemma_bounce}
Let  $\bo = \e \ll 1$, $\pdm \neq  0$  and assume $\mo$ bounces off the floor and hits the floor next before $\mt$ does.  
Then there exist sufficiently small $ \de \gg \e$ such that if  $ -\de < \ydm <0$ and  $|\phi - \frac{3\pi}{2} | \gtrsim \sqrt \de$,  
the  collision map is given by

\begin{align}\label{eq:phi}\tag{5}
\pdp&=-\pdm - \frac{2}{\sqrt{1-\ym^2}}\ydm + O\left(\frac{\e}{\de}\right) 
\end{align}
\begin{align}\label{eq:dist}\tag{6}
\yp&=\ym - \frac{2 \pi  - 2 \arccos \ds  \ym }{\pdm} \ydm + O\left(\de^{3/2}\right) + O\left(\frac{\e}{\sqrt{\de}}\right).
\end{align}
\end{lemma}

\begin{proof}
We prove (\ref{eq:phi}) in two steps. We first show that 
\[
\pdp = -\pdm - \frac{2}{\cos \phi} \ydm + O\left(\frac{\e}{\de}\right)
\]
using the expression for $\pdp$ in (\ref{eq:law}). We have,
\begin{align*}
\pdp  + \Big(\pdm &+ \frac{2}{\cos \phi} \ydm \Big)\\
&=\left( 1 - \frac{2 (1-\bo)\cos^2 \phi}{\bo + (1-\bo) \cos^2\phi} \right) \pdm+ \left( \frac{2\cos\phi}{\bo + (1-\bo) \cos^2 \phi}\right)\ydm + \left(\pdm + \frac{2}{\cos \phi} \ydm \right) \\
&=\frac{(\bo-(1-\bo)\cos^2 \phi)\pdm-(2\cos \phi )\ydm} {\bo \sin^2 \phi +\cos^2 \phi}+\left(\pdm+\frac{2}{\cos \phi} \ydm \right)\\
&=\ds {\bo}\left(\frac{\pdm+1+2{\ydm}(\frac{\sin^2{\phi}}{\cos{\phi}})}{{\bo}\sin^2{\phi}+\cos^2{\phi}}\right).
\end{align*}

For sufficiently  small $\de$, $\left|\phi-\frac{3\pi}{2}\right| \gtrsim \sqrt{\de}$ implies $\cos \phi \gtrsim \sqrt{\de}$. It follows that 

\begin{align*}
\left|\pdp  + \left(\pdm + \frac{2}{\cos \phi} \ydm\right)  \right| &\lesssim
\e \left|\frac{1+\frac{2\de}{\sqrt\de}}{\de}\right| = O\left(\frac{\e}{\de}\right).
\end{align*}

Observe from the Figure 2 that $\ds \phi = \frac{3\pi}{2} - \arccos \left( \frac{\ym}{1-\bo} \right)$. 
Thus, 
\begin{align*}
\pdp+\left(\pdm + \frac{2}{\cos \phi}\ydm \right)
=\pdp+\pdm + \frac{2}{\cos \left(\ds \frac{3\pi}{2} - \arccos \left(\frac{\ym}{1-\bo}\right) \right)}\ydm\\
=\pdp+\pdm + \frac{2}{\sqrt{1-\left(\ds \frac{\ym}{1-\bo}\right)^2}}\ydm 
=\pdp+ \pdm + \frac{2}{\sqrt{1-\ym^2}} \ydm+ R_1,
\end{align*}
where
\[\left|R_1\right| \leq \bo\left|\frac{2\ydm\ym}{\sqrt{((1-\bo)^2 - \ym^2)^3}}\right|.\]

Using that $\sqrt{(1-\bo)^2- \ym^2}=(1-\bo) \cos \phi$, we obtain 
\[\left|R_1\right| \lesssim \e\left|\frac{2\de}{((1-\bo)\sqrt \de)^3}\right|.\]

 Combining the results, we have
\begin{align*}
\left|\pdp+ \pdm + \frac{2}{\sqrt{1-\ym^2}} \ydm \right|&=\left|\pdp+\left(\pdm + \frac{2}{\cos \phi}\ydm \right) \right| + \left|R_1 \right|\\
&\lesssim\e \left(\left|\frac{1+\frac{2\de}{\sqrt\de}}{\de}\right|+ \left|\frac{2\de} {((1-\bo))^3 \de^{3/2}}\right|\right) = O\left(\frac{\e}{\de}\right).
\end{align*}

This completes the proof for (\ref{eq:phi}). \\

Let $t$ be the time between the two consecutive collisions of $\mo$. Then $\yp = \ym - \ydm t$. The angular distance that $\mo$ traveled is given by
\begin{align*}
\psi &= 2\pi - \arccos\left(\frac{\ym}{1-\bo}\right) - \arccos(\yp) 
= 2\pi  - \arccos\left(\frac{\ym}{1-\bo}\right) - \arccos\left(\frac{\ym - \ydm t}{1-\bo}\right)\\
&= 2\pi  - 2\arccos\left(\frac{\ym}{1-\bo}\right)  + R_2 =2\pi - 2\arccos \ym + R_3+ R_2,
\end{align*}
where $R_2$ and $R_3$ are the error estimates for the Taylor series expansion and are given explicitly by 
\begin{align*}
|R_2| &\leq \left|\frac{\ydm t}{\sqrt{(1-\bo)^2 - \ym^2} }\right|\\
|R_3| & \leq \left| \frac{\bo\ym}{(1-\bo)\sqrt{(1-\bo)^2 - \ym^2}} \right|.
\end{align*}
Therefore, we have
\begin{align*}
\ds \yp &= \ym - \ydm t =\ym - \ydm\left( \frac{\psi}{\pdm}\right)\\
&= \ym - \frac{\ydm}{\pdm} \left(2\pi - 2\arccos \ym + R_2 + R_3 \right)\\
&= \ym - \frac{2 \pi  - 2 \arccos \ym}{\pdm} \ydm + \frac{\ydm}{\pdm}(R_2 + R_3).
\end{align*}

Since $\mo$ can travel at most $2\pi$ between two collisions, $t$ is bounded by $\ds |t| < \frac{2 \pi}{\pd}$. Also note that $R_2$ and $R_3$ contain the factor $\ydm$ and $\bo$ respectively. We finish the proof for (\ref{eq:dist}) by computing, 
\begin{align*} 
\ds \Big|\yp  -  \ym &+ \frac{2 \pi  - 2 \arccos \ds \ym }{\pdm}\ydm \Big| 
\lesssim \left|\frac{\ydm}{\pdm}  (R_2 + R_3)\right|\\
& \lesssim \ydm^2\left|\frac{2\pi \ym}{\pdm \sqrt{(1-\bo)^2 - \ym^2}} \right|  + \bo \left|\frac{2\pi \ym}{\pdm(1-\bo)\sqrt{(1-\bo)^2 - \ym^2}} \right| \\
& \lesssim \de^2\left|\frac{2\pi}{\sqrt \de} \right|  + \e \left|\frac{2\pi}{((1-\bo))^2\sqrt\de} \right| = O\left(\de^{3/2}\right) + O\left(\frac{\e}{\sqrt\de}\right).
\end{align*}

\end{proof}

\begin{corollary}
Under the same assumptions as in Lemma \ref{lemma_bounce} with the exception 
 $\left|\phi - \frac{3\pi}{2} \right|\gtrsim {\delta^k}$ for $0\leq k \leq 1$ and $\e \ll \de^{2k}$, the variables after the collision are given by the similar equations to (\ref{eq:phi}) and (\ref{eq:dist}) but with different error terms. 

\begin{align}\tag{5a}\label{eq:phi'}
\pdp=-\pdm - \frac{2}{\sqrt{1-\ym^2}}\ydm + O\left(\frac{\e}{\de^{2k}}\right) 
\end{align}
\begin{align}\tag{6a}\label{eq:dist'}
\yp=\ym - \frac{2 \pi  - 2 \arccos \ds  \ym }{\pdm} \ydm + O\left(\frac{\de^2}{\de^k}\right) + O\left(\frac{\e}{\de^{k}}\right).
\end{align}

\end{corollary}

\begin{proof}
When computing the error terms, use $\cos \phi  \gtrsim \de^k$. 
\end{proof}


Now, we can state the adiabatic invariance theorem for the special case when the light mass hits the floor and the dumbbell is far away from the vertical position: $\phi = 3\pi/2$.
\begin{theorem}
\label{ai_theo}
Suppose right  before the collision $\pdm \neq 0$ and  $\phi-\frac{3\pi}{2} \neq 0$.  Then there is $\de >0$  such that if 
$0< \e =\de^2$,  $ -\de < \dot y_0 < 0$,
then there exists an adiabatic invariant of the dumbbell system, given by $I = |\pd|f(y)$, where $f(y) =  \pi -\arccos y $. In other words, $|\dot{\phi}_{n}| f(y_{n}) - |\dot{\phi}_{0}| f(y_{0})= O(\de)$ after $N= O(\de^{-1})$ collisions.
\end{theorem}

\begin{proof}

We prove this by finding $f(y)$ that satisfies 
\[
|\pdp| f(\yp)  - |\pdm| f(\ym)= O(\de^2).
\] 
When $\e = \de^2$ and $\de$ is sufficiently small, it follows from (\ref{eq:dist'}) that,
\[
f(\yp)=f(\ym) - \left(\frac{2 \pi  - 2 \arccos \ym }{\pdm} \ydm + O(\de^2) \right)f'(\ym). 
\]

Then, we have
\begin{align*}
|\pdp| f(\yp)&= \left| \pdm + \frac{2\ydm}{\ds \sqrt{1-\ym^2}}  +O(\de^2) \right| \left( f(\ym) -\left( \frac{2 \pi  - 2 \arccos \ym }{\pdm} \ydm + O(\de^2)\right) f'(\ym) \right)\\
&=|\pdm| f(\ym) - \left(2 \pi  - 2 \arccos\ym\right)\ydm f'(\ym) + \left|\frac{2}{\ds \sqrt{1-\ym^2}} \ydm\right| f(\ym)+ O(\de^2).
\end{align*}

Therefore, $f(y)$ satisfies $|\pdp| f(\yp) - |\pdm| f(\ym) = O(\de^2)$ provided
\[
- (2 \pi  - 2 \arccos\ym ) \ydm f'(\ym) + \frac{2}{\ds \sqrt{1-\ym^2}} \ydm f(\ym)= 0.
\]
The solution of the above equation is given by
\[
f(\ym ) = \pi - \arccos \ym.
\]

Let $\ds N = O(\de^{-1})$ and let $\dot{\phi}_{N}$ and $y_{N}$ be the angular velocity and the distance after $N^{th}$ collision. Then, we have
\begin{align*}
|\dot{\phi}_{N}| f(y_{N}) - |\dot{\phi}_{0}| f(y_{0}) &= \sum_{k=1}^{N} \left(|\dot{\phi}_{k}| f(y_{k}) - |\dot{\phi}_{k-1}| f(y_{k-1}) \right)
 \lesssim n \cdot \de^2  \lesssim \de.
\end{align*}
\end{proof}

\begin{rmk}
Adiabatic invariant has a natural geometric meaning: angular velocity times the distance traveled by the light mass 
between two consecutive collisions.
\end{rmk}

Now, we state the theorem for a realistic scenario when a rapidly rotating  dumbbell scatters
off the floor. 
\begin{theorem}
Let the dumbbell approach the floor from infinity with \mbox{$\dot \phi_- \neq 0$}. 
There exists $\de > 0$ such that  if  $0<\e =\de^2 $, $-\de < \dot y_- <0$,  $|\phi_0-\frac{3\pi}{2}|  \sim \sqrt{\delta}$
then, after $N = O(\delta^{-1})$ bounces the dumbbell will leave the floor after the final bounce by $\mo$ with $I_N =  I_0 + O(\sqrt\de)$. The adiabatic invariant is defined as 
above $I = |\dot \phi | f(y)$.
\end{theorem}
\begin{rmk}
The condition on the angle $|\phi_0-\frac{3\pi}{2}|\sim \sqrt{\delta}$ comes naturally from the following argument.  If $\yd =- \de$, the dumbbell approaching from infinity will naturally hit the floor when $y \gtrsim 1-\bo - \de$. Since $\de$ is small, this implies $|\phi_0-\frac{3\pi}{2}|\lesssim \sqrt{\de}$. If $\phi_0$ happens to be too close to $3\pi/2$, then there is no hope to
obtain adiabatic invariant and we exclude such set of initial conditions. In the limit $\delta\rightarrow 0$ the relative measure of the set where $|\phi_0-\frac{3\pi}{2}| = o(\sqrt{\de})$ tends to zero.
\end{rmk}

\begin{proof}
We will split the iterations (bounces) into two parts: before the $n^{th}$ iteration and after it, 
where $n = [\mu/\sqrt{\de}]$ and $\mu$ is sufficiently small (to be defined later). We claim that after $n$ bounces, $|\phi_{n} - \frac{3\pi}{2}| \gtrsim \sqrt[4]\de$. To prove this claim, we use energy conservation of the dumbbell system (\ref{eq:egy}), and (\ref{eq:phi'}). We have
\begin{align}\label{ydp}\tag{7}
\e(1-\e)\left(\pdm + \frac{2}{\sqrt{1-\ym^2}}\ydm + O\left(\frac{\e}{\de}\right)\right) ^2 + \ydp^2 = \e(1-\e) \pdm^2 + \ydm^2.
\end{align}
Next, 
\begin{align*}
|\ydp^2 - \ydm^2 | =\left| \e(1-\e)\pdm^2 - \e(1-\e)\left(\pdm + \frac{2}{\sqrt{1-\ym^2}}\ydm + O\left(\frac{\e}{\de}\right)\right) ^2 \right|\\
\leq\left|\e  \left(  \frac{4\ydm^2}{1-\ym^2}+  \frac{4\pdm \ydm}{\sqrt{1-\ym^2}}+ 2\pdm O\left(\frac{\e}{\de}\right) + \frac{4\ydm}{\sqrt{1-\ym^2}}O\left(\frac{\e}{\de}\right)+O\left(\frac{\e}{\de}\right)^2  \right) \right|
\end{align*}

By our assumptions, $1-\ym \gtrsim \delta$ so it follows that 
\begin{align*}
|\ydp^2 - \ydm^2|  \lesssim \de^2 \left(\frac{\de^2}{\de} + \frac{\de}{\sqrt \de}+ \de +  \frac{\de^2}{\sqrt \de} + \de^2 \right) \lesssim \de^{5/2},
\end{align*}
which implies 
\[
|\ydp - \ydm| \lesssim \delta^{3/2}.
\]


After $n=\lfloor \mu/\sqrt\de \rfloor$ bounces, $|\yd_n - \yd_0 | \leq \de/2$  if $\mu$ is sufficiently small and we still have the vertical velocity of same order, {\em i.e.} $\yd_n \sim \yd_0 \sim \de$. Then, at  the 
$n^{th}$ collision, the center of mass will be located at $y_n \lesssim 1 - \sqrt \de$, 
which will imply  $|\phi_{n} - \frac{3\pi}{2}| \gtrsim \sqrt[4]\de$. Now using Lemma 3.1, Corollary 3.2, and Theorem 3.3, we compute the error term of the adiabatic invariant under the assumption that the total number of collisions is bounded by 
$N\lesssim \delta^{-1}$ and the heavy mass does not hit the floor. 

\begin{align*}
|\dot{\phi}_{N}| f(y_N) &- |\dot{\phi}_{0}| f(y_{0}) \\
&=\sum_{k=1}^{n} \left(|\dot{\phi}_{k}| f(y_{k}) - |\dot{\phi}_{k-1}| f(y_{k-1}) \right)+ \sum_{n}^{N} \left(|\dot{\phi}_{k}| f(y_{k}) - |\dot{\phi}_{k-1}| f(y_{k-1}) \right)\\
&= \frac{\mu}{\sqrt\de}  O\left(\de)\right) + \left(\frac{C}{\de} \right) O\left(\de^{3/2}\right) =O(\sqrt\de).
\end{align*}

By the theorem proved in the next section there is indeed a uniform bound on the number of bounces.

If the heavy mass does hit the floor it can do so only once as shown in the next section. 
We claim that the corresponding change in the adiabatic invariant will be only of order $\delta$.
Indeed, using formula \eqref{eq:law} and the comment after that, we obtain 
\begin{align*}
\dot y_{+} &= -\dot y_{-} + O(\epsilon) \\
\dot \phi_{+} &= \dot \phi_{-} + O(\epsilon) + O(\delta),
\end{align*}
where subscripts $\pm$ denote the variables just after and before the larger mass hits the floor. 

Let the pairs $(y_m,\dot \phi_m)$ $(y_{m+1},\dot \phi_{m+1})$ denote  the corresponding values of 
$(y, \dot \phi)$ when the light mass hits the floor right before and after the large mass hits the floor. 
Then, since $\dot y = O(\delta)$, we find that  $y_{m+1}-y_m = O(\delta)$ and 
$\dot \phi_{m+1}-\dot \phi_m = O(\delta)$. As a consequence, 
\[
|\dot \phi_{m+1}| f(y_{m+1}) - |\dot \phi_{m}| f(y_{m}) = O(\delta)
\]
and the change in adiabatic invariant due to large mass hitting the floor is sufficiently small  $\Delta I = O(\delta)$.

\end{proof}


\section{Estimate of maximal number of collisions}

In this section, we estimate the maximal number of collisions of the dumbbell with the floor as a function of the mass ratios. As we have seen in section 2.2, on $(Y-\phi)$ plane, the dumbbell reduces to a mass point that has unit velocity and elastic reflection. We use the classical billiard result which states that the number of collisions inside a straight wedge with the inner angle $\gamma$ is given by $N_{\g}=\lceil \pi/\g\rceil$, see e.g. \cite{tabachnikov}. 

\subsection{Boundaries on $Y-\phi$ plane}
First, we discuss the properties of the boundaries of the dumbbell system on $Y-\phi$ plane. \\

When $m_1=m_2$, we have the mass ratios $\bo = \bt = 1/2$. Recall from (\ref{bd}) that the boundaries are given by 
\begin{align*}
Y={\rm max} \left\{-\sqrt{\bt/\bo}\sin\phi, \sqrt{\bo/\bt}\sin\phi \right\}= |\sin \phi | \text{ for } \phi \in [0, 2\pi] 
\end{align*}
Note that the angle between the two sine waves is $\gamma = \pi/2$.\\

When  $\mo \ne \mt$, it follows from (\ref{bd}) that the boundaries consist of two sine curves with different heights. We will assume  $\mo < \mt$,  since the case $\mt < \mo$ is symmetric. It is easy to see that generically in the limit $\mo/\mt \rightarrow 0$ most of repeated collisions will occur between two peaks of (\ref{bd1}).  In Section 4.2, we will find the upper bound for the number of collisions of the mass point to the boundaries. To start the proof, let us consider the straight wedge formed by the tangent lines to (\ref{bd1}) at $\phi=0$ and $\pi$. We call these tangent lines $\ell_0$ and $\ell_\pi$ respectively, and denote the angle of the straight wedge by $\g$, see Figure 3. Let us denote the wedge created by the union of the sine waves when $Y>0$ and the tangent lines $\ell_0$ and $\ell_\pi$ when $Y\le 0$ as the hybrid wedge.

\begin{figure}[h]
 \includegraphics[width=2.5in]{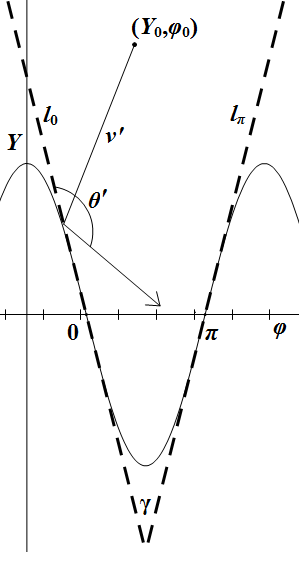}
\caption{Construction of the straight wedge and the hybrid wedge on $Y-\phi$ plane.}
\end{figure}

\subsection{The upper bound for the number of collisions}
We first introduce some notations. Denote the trajectory bouncing from the hybrid wedge by $v'$, and let the approximating trajectory bouncing from the straight wedge   by the double-prime symbols $v''$. When $v'$ or $v''$ is written with the subscript $i$, it denotes the segment of the corresponding trajectory between the $i$-th bounce and the $i+1$-st bounce. Let $\theta_i'$ be the angle from the straight wedge to $v_i'$, and $\theta_i''$ denote the angle from the straight wedge to $v_i''$ after the $i$-th collision. Define $\rho_i$ as the angle difference between the straight wedge and the curved wedge at $i$-th collision of $v_i'$. 
The trajectory will terminate when the sequence of angles terminates (due to the absence of the next bounce), or when there will be no more intersections with the straight wedge. This will happen when the angle of intersection, $\theta$, $\theta'$ and $\theta''$, with the tangent line is less than or equal to $\g$.

\begin{figure}[h]
 \includegraphics[width=5in]{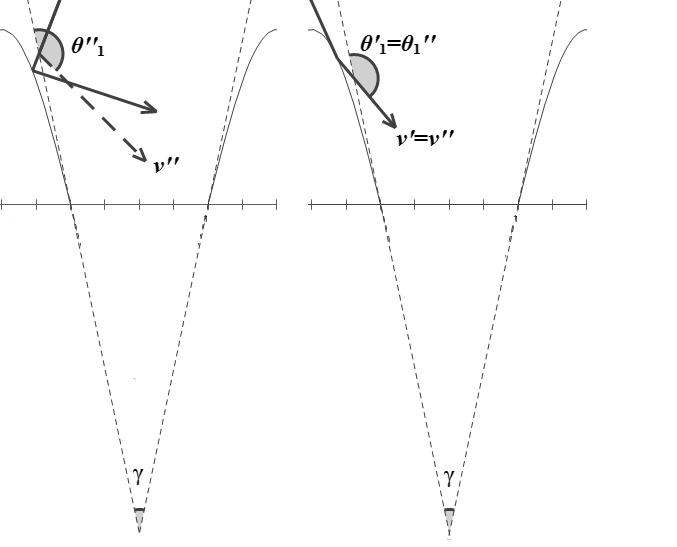}
 \caption{Two different base cases for Lemma 4.1.}
\end{figure}
\begin{lemma}
Consider the hybrid wedge and the straight wedge described above. The sequence of angles $\theta_i'', 1\le i$ will terminate after or at the same index as the sequence of angles $\theta_i', 1\le i$.
\end{lemma}

\begin{proof}
Suppose that the initial segment $v_0'$ (of the full trajectory)  crosses the straight wedge before it hits the  hybrid wedge, as shown on the right panel of Figure 4.1. Then
\begin{align*}
 \theta_1'=\pi-\g-2\rho_1'<\pi-\g=\theta_1''.
 \end{align*}
When the initial segment   $v_0'$ hits the hybrid wedge before crossing the straight wedge, as shown on the left panel,  then set $\theta_1'=\theta_1''$.
Now we can proceed by induction if $\theta_i'> \g $ and $\theta_i''>\g$ and the sequence $\theta_i'$ has not terminated. 
\begin{align*}
\theta_{i+1}'&=\theta_i'-\g-2\rho_{i+1}'\\
\theta_{i+1}''&=\theta_i''-\g
\end{align*}
which implies that $\theta_{i+1}'\le\theta_{i+1}''$.\\
Since $\theta_{i}'\le\theta_{i}''$, then $v'$ will terminate at the same time or before $v''$. 
\end{proof}

Define the bridge as the smaller sine wave created by $Y=-\sqrt{\bt/\bo}\sin\phi$ when $m_1 \ll m_2$ from $\phi=0$ to $\phi=\pi$. The union of the bridge  with the hybrid wedge will create the boundary as it is actually defined by the dumbbell dynamics.

\begin{lemma} The presence of the bridge in the hybrid wedge will increase the number of collisions of the dumbbell by at most one from the number of collisions of the dumbbell to the hybrid wedge. 

\end{lemma}
\begin{proof}
Consider the true  trajectory (denoted by $v$) that ``sees'' the bridge. Recall the  definition  of angle $\theta_i'$, which is the angle from the straight wedge to $v'_i$. Similarly, we let $\theta_i$ be the angle to $v_i$. Before $v$ intersects the bridge, by Lemma 4.1 we have
\begin{align*}
v_i &=v_i'\\
\theta_i&=\theta_i'\\
\theta_i&=\theta_{i-1}-\gamma-2\rho_{i}.
\end{align*}
Now define $\tau$ to be the angle measured from the horizontal line to the tangent line at the point where $v$ hits the bridge. Note that $\tau$ takes a positive value if the dumbbell hits the left half of the bridge, and $\tau$ takes a negative value if $v$  hits the right half of the bridge. We express $\theta_{i+1}$ after the bounce from the bridge in terms of $\theta_i$. By this convention,  the bounce from the bridge does not increase the index count but we will have to add $+1$ in the end.\\

Then, we  have 
\begin{align*}\tag{8}\label{8}
\theta_{i+1}&=\pi-\theta_i-2\tau - 2\rho_{i+1}\\
\theta_{i+1}'&=\theta_i'-\gamma-2\rho_{i+1}'
\end{align*}

We may assume that $v$ hits the bridge with non-positive velocity in $Y$. If the dumbbell hits the bridge with positive velocity in $Y$, it will continue to move in the positive $Y$ direction after reflection from  the bridge. Then, we consider the reverse trajectory to bound the number of collisions. This allows us to restrict $\theta_i$. Moreover, $v_i$ naturally hits the upper part of hybrid wedge than $v'_i$. We also assume that $v$ hits the left half of the bridge. Otherwise, we can reflect the orbit around the vertical line passing through the middle point of the bridge. 

Utilizing the above arguments, we have the inequalities

\begin{align*}\tag{9} \label{9}
 \frac{\pi+\g}{2}&<\theta_i \\
0<\tau &< \frac{\g}{2}\\
\rho_{i+1}' &< \rho_{i+1}.
\end{align*}

It is straightforward to verify that (\ref{8}) and (\ref{9}) imply $\theta_{i+1} \leq \theta_{i+1}'$.  From the $i+2$-nd bounce, if $\theta_i$ has not terminated, we can apply induction argument similar to the proof in Lemma 4.1.  We have the base case 
\begin{align*}
\theta_{i+1} &\leq \theta_{i+1}'\\
\rho_{i+1} & \geq \rho_{i+1}'.
\end{align*}

Note that $\rho$'s indicate the relative position of a collision point in the hybrid wedge. That is, if $\rho_{i+1}  \geq \rho_{i+1}'$, then the starting point of $v_{i+1}$ is located at or above that of $v_i$. Since $\theta_{i+1} \leq \theta_{i+1}'$ and $v_{i+1}$ starts above $v_{i+1}'$, we know $v_{i+2}$ will start on the hybrid wedge higher than $v_{i+2}'$. This implies 
\[
\rho_{i+2} \leq  \rho_{i+2}'.
\]
Then using the recursive relationship, 
\begin{align*}
\theta_{i+2}=\theta_{i+1}-\gamma-2\rho_{i+2}\\
\theta_{i+2}'=\theta_{i+1}'-\gamma-2\rho_{i+2}',
\end{align*}
 we obtain $\theta_{i+2} \leq \theta_{i+2}'$. By induction $\theta_i \leq \theta_i '$ for all $i$. Taking into account the bounce on the bridge, we conclude that the number of bounces of $v$ will increase at most by one relative to that of $v'$. Note that in most cases, the number of bounces of $v$ will be less than the number of bounces of $v'$. 
\end{proof}

Now we are ready to prove the main theorem.
\begin{theorem}
The number of collisions of the dumbbell is bounded above by $\ds N_{\g}=\big \lceil \pi/\g\big\rceil+1$, where $\ds \g = \pi - 2 \arctan \sqrt{\bt /\bo}$. 
\end{theorem}

\begin{proof}
When $\mo =\mt$, as we have found in the previous section 4.1, the boundaries form identical hybrid wedges which intersect at $\pi/2$. Using Lemma 4.1, we conclude that the upper bound for the number of collisions is $\lceil\pi/\g\rceil= 2$, which is less than $N_{\g}=3$.\\
When $\mo < \mt$, we consider the true boundaries which consist of a hybrid wedge with the bridge. Using Lemma 4.1 and Lemma 4.2, we conclude that the the maximal number of collisions to the true boundary is bounded above by $N_\g=\lceil\pi/\g\rceil+1$. Since $\g = \pi-2\arctan\sqrt{\bt/\bo}$, this completes the proof. 
\end{proof}

\section*{Acknowledgment}
The authors acknowledge support from National Science Foundation 
grant DMS 08-38434 ”EMSW21-MCTP: Research Experience for Graduate Students." 
YMB and VZ were also partially supported by NSF grant  DMS-0807897. The authors 
would also like to thank Mark Levi for a helpful discussion.

\end{document}